\documentclass[11pt, a4paper]{article}

\usepackage{amsmath, amsfonts, amssymb, amsthm, natbib}

\DeclareMathSymbol{\leqslant}{\mathalpha}{AMSa}{"36} 

\DeclareMathSymbol{\geqslant}{\mathalpha}{AMSa}{"3E} 

\DeclareMathSymbol{\eset}{\mathalpha}{AMSb}{"3F}     

\renewcommand{\leq}{\;\leqslant\;}                   

\renewcommand{\geq}{\;\geqslant\;}                   

\newcommand{\supp}{\mathrm{\,supp\,}}

\newcommand{\range}{\mathrm{range\,}}

\newcommand{\cA}{\mathcal{A}}
\newcommand{\cB}{\mathcal{B}}
\newcommand{\cC}{\mathcal{C}}

\newcommand{\cK}{\mathcal{K}}
\newcommand{\cL}{\mathcal{L}}
\newcommand{\cM}{\mathcal{M}}

\newcommand{\cO}{\mathcal{O}}

\newcommand{\cU}{\mathcal{U}}

\newcommand{\Om}{\Omega}

\newcommand{\eps}{\varepsilon}
\newcommand{\lam}{\lambda}

\newcommand{\vphi}{\varphi}

\newcommand{\lfl}{\lfloor}
\newcommand{\rfl}{\rfloor}

\newcommand{\dto}{\downarrow}

\newcommand{\IP}{\mathbb{P}}
\newcommand{\II}{\mathbb{I}}

\newcommand{\IN}{\mathbb{N}}

\newcommand{\IZ}{\mathbb{Z}}

\newcommand{\IR}{\mathbb{R}}
\newcommand{\IE}{\mathbb{E}}
\newcommand{\iN}{\in\IN}
\newcommand{\iZ}{\in\IZ}
\newcommand{\iR}{\in\IR}

\renewcommand{\le}{\leq}
\renewcommand{\ge}{\geq}

\newcommand{\be}{\begin{eqnarray*}}
\newcommand{\ee}{\end{eqnarray*}}
\newcommand{\ben}{\begin{eqnarray}}
\newcommand{\een}{\end{eqnarray}}

\newcommand{\citeasnoun}{\cite}

\theoremstyle{plain}
\newtheorem{theo}{Theorem}[section]
\newtheorem{lemma}[theo]{Lemma}
\newtheorem{propo}[theo]{Proposition}

\theoremstyle{definition}
\newtheorem{defi}[theo]{Definition}

\newtheorem{remark}[theo]{Remark}
\newtheorem{example}[theo]{Example}

\renewenvironment{proof}[1][] {\noindent{\bf Proof#1. }}{\hspace*{\fill}$\square$\medskip\par}

\begin{document}

\vglue20pt
\centerline{\huge\bf The high resolution }
\medskip

\centerline{\huge\bf vector quantization problem }
\medskip

\centerline{\huge\bf with Orlicz norm distortion}

\bigskip
\bigskip

\centerline{by}
\bigskip
\medskip

\centerline{{\Large S. Dereich and  C. Vormoor}}

\bigskip

\begin{center}
\it  Fachbereich Mathematik und Informatik\\
Philipps--Universit\"at Marburg\\
Hans--Meerwein Stra\ss e\\
D-35032 Marburg\\
E-mail: dereich@mathematik.uni-marburg.de, cvormoor@skandia.de\\
\end{center}

\bigskip
\bigskip
\bigskip

{\leftskip=1truecm
\rightskip=1truecm
\baselineskip=15pt
\small

\noindent{\slshape\bfseries Summary.}
We derive a high-resolution formula for the quantization problem under Orlicz norm distortion.
In this setting, the optimal point density  solves a
variational problem which comprises a function $g:\IR_+\to[0,\infty)$
characterizing the quantization complexity of the underlying Orlicz
space. Moreover, asymptotically optimal codebooks induce a tight  sequence of
empirical measures. The set of possible accumulation points is characterized and in most cases it consists of a single element. In that case, we find convergence as in the classical setting.

\bigskip


\bigskip


}

\section{Introduction}

For  $d\iN$, consider an $\IR^d$-valued random vector $X$ (\emph{the original}) defined on some probability space $(\Om,\cA,\IP)$, and denote by $\mu=\cL(X)$ the law of $X$.

We consider the quantization problem, that is for a given natural number $N\iN$
and a loss function $\rho:\IR^d\times \IR^d\to[0,\infty)$ we ask for a codebook
$\cC\subset \IR^d$ consisting of at most $N$ elements and minimizing the
average loss
$$
\IE \rho(X,\cC),
$$
where $\rho(x,A)=\inf_{y\in A} \rho(x,y)$, for all $x\iR^d$ and $A\subset\IR^d$.
The quantization problem arises naturally when discretizing analog signals, and it first gained practical importance in the context of pulse-code-modulation. Research on it started in the 1940's and one finds numerous articles dedicated to the study of this problem in the engineering literature. For an overview on these developments, one may consult \citeasnoun{GrNeu98} (see also \citeasnoun{CoTho91} and \citeasnoun{GerGra92}).
The quantization problem is also related to  numerical integration  \cite{PaPri04b}, and, more recently, the mathematical community became attracted by the field.
In the last years a number of new publications appeared treating finite dimensional as well as infinite dimensional signals (see
for instance \citeasnoun{GraLu00}, \citeasnoun{Gru04} for vector quantization
and \citeasnoun{LuPa04}, \citeasnoun{DFMS03} for functional quantization).


In this article, we consider asymptotic properties of the quantization
problem when the size $N$ of the codebook tends to infinity, the
\emph{high-resolution quantization problem}. First asymptotic
formulae for vector quantization were found by \cite{Zad66, Zad82}
and \citeasnoun{BuWi82}.

%
%
In the classical setting (\emph{norm based distortion}), one considers
$$
\rho(x,y)=\|x-y\|^p
$$
for some norm $\|\cdot\|$ on $\IR^d$ and a moment $p\ge1$. As long as the distribution
$\mu:=\cL(X)$ has thin tails (in an appropriate sense) one can describe
asymptotically optimal codebooks via an optimal point density function: the
empirical measures associated to optimal codebooks $\cC(N)$ of size $N$
$$
\frac1N \sum_{\hat x\in\cC} \delta_{\hat x}
$$
converge to a continuous probability measure that has density proportional to
$\bigl(\frac{d\mu_c}{d\lam^d}\bigr)^{d/(d+p)}$. The density will be called \emph{optimal point density}. Here  and elsewhere, $\lam^d$ denotes $d$-dimensional Lebesgue measure and $\mu_c$ denotes the absolutely continuous part of $\mu$ w.r.t.\ $\lam^d$. When considering point densities we will always assume that $\mu_c$ does not vanish.  Optimal codebooks for the
uniform distribution  can then be used to define asymptotically optimal codebooks for
general $X$: roughly speaking, one partitions the space $\IR^d$ into
appropriate cubes and chooses in each cube an optimal codebook for the uniform
distribution of appropriate size (according to the optimal point density). In
particular, the non-continuous part of $\mu$ has no effect on the asymptotic
problem. The concept of a point density  will play a crucial role in
the following discussion. Its importance in the classical setting  was first
conjectured by Lloyd and Gersho (see \citeasnoun{Ger79}). First rigorous proofs are due to
\citeasnoun{Buck84}. For a recent account on the theory of high resolution
quantization and point density functions one may consult the monograph by
\citeasnoun{GraLu00}.

Nowadays, the asymptotic quantization problem is well understood for loss
functions $\rho$ that are shift invariant and look locally like a power of a
norm based distance, that is
$$
\rho(x,y)=\rho(x-y) \ \text{ and } \ \rho(x)=\|x\|^p+ o(\|x\|^p) \text{ as }x\to0
$$
for some norm $\|\cdot\|$ on $\IR^d$ and a power $p\ge1$ (see \citeasnoun{DGLP05}).  Here and thereafter $o$ and $\cO$ denote the
Landau symbols.  For all these distortion measures  one regains the same
optimal point density  as for the corresponding norm based distortion
measures. Since the optimal point density only depends  on the  behavior of $\rho$ close  to~$0$,   quantization schemes based on this density may show bad performance for moderate $N$. This will occur, when for the original $X$ and the approximation $\hat X$ the distances $\rho(X,\hat X)$ and $\|X-\hat X\|^p$ differ significantly.


As a generalization of the above setting, we suggest to use Orlicz norms as a
measure for the loss inferred when approximating the original $X$ by the
closest point in a codebook~$\cC$.
One reason for this is that optimal codebooks for general distortions are also optimal codebooks for certain Orlicz norm distortions. For moderate $N$ the  optimal point density of the corresponding Orlicz norm distortion seems to be the favorable choice as basis for the construction of good codebooks since it incorporates   $\rho$ as a whole and not only its asymptotics  in zero.

Let us introduce the main notation.
Let $\vphi:[0,\infty)\to[0,\infty)$ be an increasing, left
continuous function with $\lim_{t\dto0} \vphi(t)=0$. Note that this implies
that $\vphi$ is lower semicontinuous. We assume that $\vphi\not=0$,  let
$E=(\IR^d,\|\cdot\|)$ denote an arbitrary Banach space, and denote by $d(\cdot,\cdot)$ the associated distance.  For any $\IR^d$-valued
r.v.\ $Z$, the Orlicz norm $\|\cdot\|_{\vphi}$ is defined as
$$
\|Z\|_{\vphi}=\inf\Bigl\{t\ge0: \IE \,\vphi\Bigl(\frac{\|Z\|}{t}\Bigr) \le1\Bigr\},
$$
with the convention that the infimum of the empty set is equal to infinity. Actually, the left continuity of
$\vphi$ together with monotone convergence imply that the infimum is attained,
whenever the set is nonempty. We set
$$
 L^\vphi(\IP)=\{Z : Z \ \IR^d\text{-valued r.v.\ with }\|Z\|_\vphi<\infty\}.
$$
Note that $\|\cdot\|_\vphi$ defines a norm on $L^\vphi(\IP)$ when $\vphi$ is
convex, whereas otherwise, the triangle inequality does not hold. For our
analysis we do not require that $\vphi$ be convex. Nevertheless, with  slight
misuse of notation, we will allow ourselves to call $\|\cdot\|_\vphi$ an Orlicz
norm. Choosing $\vphi(t)=t^p$, $p\ge1$, yields the usual $L^p(\IP)$-norm, which
will be denoted by $\|\cdot\|_{p}$.

For $N\ge1$, we consider the \emph{quantization error} given by
$$
\delta(N|X,\vphi)=\inf_{\hat X} \|X-\hat X\|_\vphi,
$$
where the infimum is taken over all r.v.'s $\hat X$
(\emph{reconstructions}) satisfying the range constraint $|\range(\hat
X)|\le N$. In the case where $\vphi(t)=t^p$ for a $p>0$, we write briefly  $\delta(N|X,p)$ for the corresponding quantization error. 

Let us compare the Orlicz norm distortion with the classical setting.
Suppose that $\hat X$ is an optimal $N$-point quantizer under the distortion $\rho(x,y)=f(\|x-y\|)$, where  $f$ is a  left continuous and strictly increasing function with $f(0)=\lim_{t\dto 0} f(t)=0$.
Then one can easily verify that $\hat X$ is also an optimal $N$-point quantizer in the Orlicz norm setting when choosing $\vphi(t)= f(t)/\Delta$ and $\Delta=\IE[\rho(X,\hat X)]$. Thus optimal quantizers in the classical setting correspond to optimal quantizers in the Orlicz-norm setting. As we will see later, in most cases each choice of $\Delta$ (or $\vphi$) leads to a unique optimal point density function. Thus the optimal point density function in the classical setting is replaced by a whole family of densities: good descriptions are now obtained by choosing the parameter $\Delta$ accordingly. We believe that  the optimal Orlicz
point density is a favourable description of good codebooks for moderate $N$.
Let us illustrate this in the case where $f(t)=\vphi(t)=\exp(t)-1$ (meaning that $\Delta=1$).
Whereas the Orlicz norm  point density attains rather large values even at points where the density $\frac{d\mu_c}{d\lam^d}$ is very small (see Example \ref{ex11}), the classical setting neglects the growth of $f$ and one retrieves the optimal density of the norm based distortion to the power $1$.




In our notation, the asymptotics under the classical $L^p$-norm distortion reads as follows (see \citeasnoun{GraLu00}):
Let $p\ge1$, $U$ denote a uniformly distributed r.v.\ on $[0,1)^d$, and set
$$
q(E,p)= \inf_{N\ge1} N^{1/d} \,\delta(N|U,p).
$$
If for some $\tilde p>p$, $\IE \|X\|^{\tilde p}<\infty$ (\emph{concentration assumption}), then
\begin{align}\label{eq0322-1}
\lim_{N\to\infty} N^{1/d} \,\delta(N|X,p)= q(E,p)\, \Bigl\|\frac{d\mu_c}{d\lam^d}\Bigr\|_{L^{d/(d+p)}(\IR^d)} ^{1/p},
\end{align}
where $\mu_c$ denotes the absolutely continuous part of $\mu$.

%

The analysis of the Orlicz norm setting is based on the concept of a
point allocation density.
We shall see that an optimal point allocation density is given as a minimizer of a variational problem.
Unfortunately, in this context the minimization problem cannot be solved in closed form. We will prove the existence and a dual characterization of the solution.
The quantity $q(E,p)$ corresponds in the general setting to a convex decreasing function $g:(0,\infty)\to [0,\infty)$ which may be defined via
\begin{align}\label{eq0817-1}
g(\eta)=\lim_{N\to\infty} \inf_{\cC(N)} \IE \,\vphi\bigl((N/\eta)^{1/d} \,d(U,\cC(N)) \bigr),
\end{align}
where the infima are taken over all finite sets $\cC(N)\subset \IR^d$ with $|\cC(N)|\le N$ and $U$ denotes a uniformly distributed r.v.\ on the unit cube $[0,1)^d$.  Moreover, we set $g(0)=\liminf_{\eta\dto0} g(\eta)\in[0,\infty]$.
The function $g$ depends on the Banach space and the Orlicz norm. It
will be analyzed in Section \ref{sec2}. $g$ can be represented as a
particular integral in the situations where $E=l_\infty^d$ or $E=l_2^2$. If $\vphi$ induces the $L^p(\IP)$-norm (i.e.\ $\vphi(t)=t^p$), then due to (\ref{eq0322-1})
$$
g(\eta)=q(E,p)^p \, \eta^{-p/d}.
$$

If the measure $\mu$ is not compactly supported, our analysis relies on a concentration property. As in the $L^p(\IP)$-setting, this can be done by assuming the finiteness of an integral $\IE \Psi(\|X\|)$ for some function $\Psi$ satisfying a growth condition (Condition (G), see Definition \ref{defi0322-1}).
Let us state the main theorem.

\begin{theo}\label{theo0323-2}
Assume that $\Psi$ satisfies the growth condition (G) and that $\IE \Psi(\|X\|)<\infty$. Then
$$
\lim_{N\to\infty} N^{1/d}\,\delta(N|X,\vphi)= I^{1/d},
$$
where $I$ is the finite minimal value in the point allocation problem. It is given by
\begin{align}\label{eq0519-1}
I=\inf_\xi \int \xi(x)\,dx,
\end{align}
where the infimum is taken over all non-negative Lebesgue integrable functions $\xi$ with
$$
\int_{\IR^d} g(\xi(x))\,d\mu_c(x)\le1.
$$
Alternatively, one can represent $I$ by the dual formula
$$
I=\sup_{\kappa>0} \frac1\kappa\Bigl(\int_{\IR^d} \bar g\bigl(\frac{\kappa}{h(x)}\bigr)\,d\mu_c(x)-1 \Bigr),
$$
where $\bar g(t)=\inf_{\eta>0}[g(\eta)+ \eta t]$ and $h(x)=\frac{d\mu_c}{d\lam^d} (x)$.
Moreover, one has $I>0$ if and only if $$\mu_c(\IR^d) \sup_{t\ge0} \vphi(t)> 1,$$
where we use the convention that $0\cdot \infty =0$.
\end{theo}

\begin{remark}\begin{itemize}\item
If we choose $\vphi(t)=t^p$ in the former theorem, we obtain the classical result,  since  $\Psi(t)=t^q$ with $q>p$ satisfies the growth condition~(G), see Example~\ref{exa0817-1} for an even weaker assumption.
\item The point allocation problem is studied in Section~\ref{sec5}. In particular, representations for optimizers can be found in Theorem~\ref{theo0315-1} and Remark~\ref{re0406-1}.
\end{itemize}
\end{remark}

We keep $I$ as the minimal value in the point allocation problem given by (\ref{eq0519-1}).
If $I$ is strictly bigger than $0$, we denote by $\cM$ the set of probability measures on the Borel sets of $\IR^d$ associated to the minimizers of the point allocation problem, i.e.,
\begin{align}\label{eq0525-1}
\cM=\Bigl\{ \nu: \frac{d\nu}{d\lam^d} = \bar\xi, \ \int_{\IR^d} g(I\,\bar\xi(x))\,d\mu_c(x)=1, \ \int_{\IR^d} \bar\xi(x)\,dx=1\Bigr\}.
\end{align}

\begin{theo}\label{th0519-1}
Assume that $I\in(0,\infty)$ and denote by $\cC(N)$, $N\iN$, asymptotically optimal  codebooks of size $N$, that is
$$
\limsup_{N\to\infty} N^{1/d}\,\| d(X,\cC(N))\|_{\vphi}\le I^{1/d}.
$$
Then the empirical measures $\nu_N$ given by
$$
\nu_N=\frac1N\sum_{\hat x\in\cC(N)} \delta_{\hat x}
$$
form a tight sequence of probability measures, and any accumulation point of $(\nu_N)_{N\iN}$ lies in $\cM$. If $g$ is strictly convex, then the set $\cM$ contains exactly one measure $\nu$, and
$$
\lim_{N\to\infty} \nu_N=\nu\qquad\text{weakly.}
$$
\end{theo}


As an example we present implications of our results for the standard normal distribution under a particular $\vphi$ growing exponentially fast:

\begin{example}\label{ex11}
Let $\mu$ denote the standard normal distribution,
$(E,\|\cdot\|)=(\IR,|\cdot|)$ and $\vphi:[0,\infty) \to [0,\infty),\
x \mapsto  \exp(x)-1$. Then $\Psi(x)=\exp(x^{3/2})$ satisfies the
growth condition (G) (see Example \ref{exa0817-1}), and Theorem
\ref{theo0323-2} is thus applicable.
Moreover, (see Remark \ref{re1006-1} and Lemma \ref{le1006-1})
$$
g(\eta)=2 \int_0^{1/2} \vphi(t/\eta)\,dt= \eta e^{1/(2\eta)} -2\eta-1.
$$
Note that $g$ ist strictly convex, so that there has to exist a unique normalized optimal point density $\bar\xi$.
In order to approximate the optimal value for $\kappa$, we used numerical methods to obtain $\kappa\approx 0.699$.
Due to (\ref{eq0526-3}) and (\ref{eq1006-1}), the optimal point density is given by
$$
\xi(x)= (-g')^{-1}\bigl(\frac{\kappa}{h(x)}\bigr), \qquad x\iR,
$$
and $I= \int_\IR \xi(x)\,dx\approx 2.88$ .
Next, elementary calculus gives
$$
(-g')^{-1}(t)=\frac12 \frac1{\log(t)-\log(2\log t)+o(1)}\quad \text{as }t\to\infty
$$
so that the normalized point density $\bar \xi=\xi/I$ satisfies
$$
\bar\xi(x)\sim\frac1I\frac1{x^2 }\quad\text{as }|x|\to\infty.
$$
Hence,  $\bar\xi$ decays to zero much more slowly than in the classical setting.
\end{example}

%
%
%

The article is outlined as follows.
In Section \ref{sec2}, we begin with an analysis of the function $g$. 
In Section \ref{sec3} we construct asymptotically good codebooks based
on a given point allocation measure. Up to this stage, we are
restricting ourselves to absolutely continuous measures with compact support.
In Section \ref{sec4}, we turn things around and prove a lower bound based on a given point density measure. This bound implies the lower bound in Theorem \ref{theo0323-2} and proves a part of Theorem \ref{th0519-1}.
The estimates of Sections \ref{sec3} and \ref{sec4} lead to the
variational problem characterising the point density, and this is treated in Section \ref{sec5}.
In the last two sections, we treat the upper bounds in the quantization problem for singular and non-compactly supported measures. In particular, we derive a concentration analog which guarantees that the quantization error is of order $\cO(N^{-1/d})$.
Finally, we combine the estimates and prove the general upper bound in Theorem \ref{theo0323-2}.

It is convenient to use the symbols $\sim$, $\lesssim$ and $\approx$.
We write   $f \sim g$ iff
$\lim \frac fg = 1$,
while $f \lesssim g$ stands for $\limsup \frac fg \le 1$.
Finally, $f\approx g$ means $ 0< \liminf \frac fg \le \limsup \frac fg <\infty$ .

\section{First estimates for the uniform distribution}\label{sec2}

In this section, $X$ denotes a uniformly distributed r.v.\ on $[0,1)^d$.
For $\eta>0$ and $N\ge1$, we consider
\begin{align}\label{eq0317-1}
f_N(\eta):=\inf_{\cC(N)} \IE \,\vphi\Bigl( (N/\eta)^{1/d}\, \min_{\hat x\in\cC(N)} \|X-\hat x\| \Bigr),
\end{align}
where the infimum is taken over all codebooks $\cC(N)\subset \IR^d$ of size $\lfl N\rfl$. Here and elsewhere, $\lfl N\rfl$ denotes the largest integer smaller or equal to $N$.
By a straightforward argument, the lower semicontinuity of $\vphi$ implies that the function
$$
\bigl(\IR^d\bigr)^{\lfl N\rfl} \ni(\hat x_1,\dots,\hat x_{\lfl N\rfl})\mapsto \IE\,\vphi\Bigl( (N/\eta)^{1/d}\, \min_{i=1,\dots,\lfl N\rfl} \|X-\hat x_i\| \Bigr)\in [0,\infty)
$$
is also lower semicontinuous. In the minimization problem
(\ref{eq0317-1}), it suffices to allow for codebook entries that are elements of  a sufficiently large compact set. So the lower semicontinuity implies the existence of an optimal codebook.
We usually denote by $\hat X^{(N)}$ or $\hat X^{(N,\eta)}$ an optimal reconstruction attaining at most $N$ different values, that is $\hat X=\hat X^{(N)}$ is a minimizer of
$$
\IE\,\vphi\Bigl( (N/\eta)^{1/d}\,  \|X-\hat X\| \Bigr),
$$
among all r.v.'s satisfying the range constraint $|\range(\hat X)|\le N$.
Now define the function $g$ by
$$
g(\eta)=\inf_{N\ge1} f_N(\eta),\qquad \eta>0.
$$
We start with a derivation of the structural properties of $g$. In particular, we show the validity of (\ref{eq0817-1}).

\begin{theo}\label{le0223-1}
The function $g:\IR_+ \to[0,\infty)$ is decreasing and convex,
and satisfies $
\lim_{\eta\to\infty} g(\eta)=0$ and $\lim_{\eta\dto0} g(\eta)=\sup_{t\ge0} \vphi(t)$. Moreover, for $\eta>0$,
\begin{align}\label{eq0223-3}
\lim_{N\to\infty} \IE \,\vphi\Bigl( (N/\eta)^{1/d}\, \|X-\hat X^{(N)}\| \Bigr)= g(\eta).
\end{align}
\end{theo}

We will sometimes use the convention $g(0)=\lim_{\eta\dto0} g(\eta)$. Note that $g(0)$ is finite iff $\vphi$ is bounded. Moreover, we set $f_N(\eta)=\infty$ for $N\in[0,1)$.

\begin{remark}\label{re1006-1}
In general, computing the function $g$ explicitly constitutes a hard
problem. However, as for the classical $L^q$-norm distortion, one can
calculate $g$ when $E=\IR^d$ is endowed with supremum-norm, and in the case where $E$ is the two dimensional Euclidean space.
In such cases, the same lattice quantizers can be used to construct
asymptotically optimal codebooks and to compute the function $g$. The case of the supremum-norm is trivial since the unit ball is space filling. 
\end{remark}

\begin{lemma} \label{le1006-1} Let $U$ be uniformly distributed  on a
  centered regular hexagon $V$ in $\IR^2$ having unit area, and assume that $E$ is the $2$-dimensional Euclidean space. One has
$$
g(\eta)=\IE \,\vphi\bigl(\eta^{-1/2}\,\|U\|\bigr).
$$
\end{lemma}

The proof is similar as in the classical setting and therefore ommitted, see \citeasnoun[Theorem 8.15]{GraLu00} and \citeasnoun{Fej72}.

In order to prove Theorem \ref{le0223-1}, we use the inequality below. It is essentially a consequence of the self similarity of $X$.

\begin{propo} \label{propo0307-1} Let $M\in\{k^d:k\iN\}$, $\eta_2>\eta_1>0$ and let $\eta=\alpha \eta_1+\beta \eta_2$ be a convex combination of $\eta_1$ and $\eta_2$.
 Then for any $N\ge1$ one has
\begin{align}\label{eq0223-1}
f_N(\eta)
\le a(M)\, f_{N_1/M}(\eta_1)+b(M)\, f_{N_2/M}(\eta_2),
\end{align}
where
\begin{itemize}
\item $N_1=\frac{\eta_1}{\eta} N$ and $N_2=\frac{\eta_2}{\eta} N$,
\item $b(M)=\lfl \beta M\rfl/M$ and $a(M)=1-b(M)$.
\end{itemize}
Additionally, one has for $N\ge1$ and $\eta>0$,
\begin{align}\label{eq0223-2}
f_N(\eta)\le f_{N/M}(\eta).
\end{align}
\end{propo}

\begin{proof}  Fix $N\iN$ and let
 $\eta_1,\eta_2,\alpha,\beta, N_1, N_2, M=k^d$ be as in the proposition.
Let $C_1=[0,1/k)^d$.
We decompose the cube $[0,1)^d$ into an appropriate union
$\bigcup_{i=1}^M C_i$ of disjoint sets $C_1,\dots,C_M$, where each set
$C_2,\dots,C_M$ is a translate of $C_1$. Moreover, let $X_i$ denote a uniformly distributed r.v.\ on $C_i$.
Since $\cU([0,1)^d)=\frac1M \sum_{i=1}^M \cU(C_i)$, one has, in analogy to Lemma 4.14 in \citeasnoun{GraLu00},
\begin{align*}
f_N(\eta)=\IE \,\vphi\Bigl( (N/\eta)^{1/d}\, \|X-\hat X^{(N)}\| \Bigr) \le \frac 1M  \sum_{i=1}^M \IE \,\vphi\Bigl( (N/\eta)^{1/d}\, \|X_i-\hat X_i^{(\tilde N_i)}\| \Bigr)
\end{align*}
for any $[1,\infty)$-valued sequence $(\tilde N_i)_{i=1,\dots,M}$ with $\sum_i \tilde N_i\le N$.
Here, $\hat X_i^{(\tilde N_i)}$ denotes an optimal quantizer for $X_i$ among all quantizers attaining at most $\tilde N_i$ different values. Since the distributions $\cU(C_i)$ can be transformed into $\cU(C_1)$ through a translation, one obtains
$$
f_N(\eta)\le \frac 1M  \sum_{i=1}^M \IE \,\vphi\Bigl( (N/\eta)^{1/d}\, \|X_1-\hat X_1^{(\tilde N_i)}\| \Bigr).
$$
Self similarity ($\cL(X_1)=\cL(\frac1k X)$) then implies that
$$
f_N(\eta)\le \frac 1M  \sum_{i=1}^M \IE \,\vphi\Bigl( (N/\eta)^{1/d}\,\frac1k\, \|X-\hat X^{(\tilde N_i)}\| \Bigr).
$$
Note that the assertion of the proposition is trivial if $N_1/M<1$, so
that we may assume $N_2/M\ge N_1/M\ge1$.
We now choose for $\lfl \beta M\rfl$ indices $\tilde N_i= N_2/M$ and
for $M- \lfl \beta M\rfl$ indices $\tilde N_i= N_1/M$. Then
$\sum_i\tilde N_i\le N$, so that
\begin{align*}\begin{split}
f_N(\eta)&\le a(M) \,\IE \,\vphi\Bigl( (N/\eta)^{1/d}\,\frac1k\, \|X-\hat X^{(N_1/M)}\| \Bigr)\\
& + b(M)\,\IE \,\vphi\Bigl( (N/\eta)^{1/d}\,\frac1k\, \|X-\hat X^{( N_2/M)}\| \Bigr)\\
&=a(M)\, f_{N_1/M}(\eta_1)+b(M)\, f_{N_2/M}(\eta_2),
\end{split}\end{align*}
where $b(M)=\lfl \beta M\rfl /M$ and $a(M)=1-b(M)$. Analogously,
setting $\tilde N_i= N/M$ for $i=1,\dots,M$, we obtain that
$f_N(\eta)\le f_{N/M}(\eta)$.
\end{proof}

\begin{proof}[ of Theorem \ref{le0223-1}]
Obviously $g$ is decreasing. First we prove that for arbitrary $\eta>0$,
\begin{align}\label{eq0307-2}
g(\eta)\le \limsup_{N\to\infty} f_N(\eta)\le g_-(\eta).
\end{align}
Fix $\eps>0$ and choose $\eta_0\in(0,\eta)$ so that $g(\eta_0)\le g_-(\eta)+\eps/2$. Moreover, fix $N_0\ge1$ with $f_{N_0}(\eta)\le g(\eta_0)+\eps/2$.
For $N\ge N_0$, we decompose $N$ into  $N=N_0 \,k^d+\tilde N$, where $k=k(N)\iN$ and $\tilde N=\tilde N(N) \iN_0$ are chosen so that $N<(k+1)^d N_0$.
Then
\begin{align*}
f_N(\eta)&=\IE\, \vphi\Bigl( (N/\eta)^{1/d}\, \|X-\hat X^{(N)}\| \Bigr)\\
&\le \IE\, \vphi\Bigl( \bigl(N_0\,k^d /(N_0\,k^d \eta /N)\bigr)^{1/d}\, \|X-\hat X^{(N_0\,k^d)}\| \Bigr)=f_{N_0\,k^d}({\eta N_0\,k^d /N}),
\end{align*}
and inequality (\ref{eq0223-2}) implies that for $M=M(N)=k^d$:
$$
f_N(\eta)\le f_{N_0}({\eta N_0 M/N}).
$$
Note that $N_0 k^d \le N<N_0 (k+1)^d$, hence: $\lim_{N\to\infty} \eta N_0 k^d/N =\eta$.
Consequently, there exists $N_1\ge N_0$ such that for all $N\ge N_1$
one has: $\eta N_0 M/N\ge \eta_0$, and
$$
f_N(\eta)\le f_{N_0}({\eta N_0 M/N})\le f_{N_0}({\eta_0})\le g_-(\eta)+\eps
$$
for all $N\ge N_1$.
Since $\eps>0$ was arbitrary  statement (\ref{eq0307-2}) follows.

We now prove that $g_-$ is convex. Let $\eta_2>\eta_1>0$ and let $\eta=\alpha \eta_1+\beta \eta_2$ be a convex combination of $\eta_1$ and $\eta_2$ and suppose that $g_-(\eta_1)$ is finite. Fix $k\iN$ and let $M=k^d$, $a(M)$ and $b(M)$ be as in Proposition \ref{propo0307-1}. Moreover, for given $N\iN$ we let $N_1=N_1(N)$ and $N_2=N_2(N)$ be as in the previous proposition. Then inequality  (\ref{eq0223-1}) implies that
\begin{align*}
f_N(\eta)
\le a(M)\, f_{N_1/M}({\eta_1}) +b(M)\, f_{N_2/M} ({\eta_2}).
\end{align*}
Therefore, formula (\ref{eq0307-2}) and the left continuity of $g_-$ give
$$
g(\eta)\le \limsup_{N\to\infty}f_N(\eta) \le a(M)\, g_-(\eta_1) +b(M)\, g_-(\eta_2).
$$
Recall that $M\in\{k^d:k\iN\}$ was arbitrary. Since $\lim_{M\to\infty} a(M)=\alpha$ and $\lim_{M\to\infty} b(M)=\beta$ we conclude that
$$
g(\eta) \le \alpha\, g_-(\eta_1)+\beta\, g_-(\eta_2).
$$
For the general statement, observe that
$$
g_-(\eta)=\lim_{\delta\dto 0} g(\eta-\delta) \le \limsup_{\delta\dto0} \bigl[ \alpha\, g_-(\eta_1-\delta)+\beta\, g_-(\eta_2-\delta)\bigr]= \alpha\, g_-(\eta_1)+\beta\, g_-(\eta_2).
$$
Consequently, $g_-$ is convex, and a forteriori it is continuous.  Therefore, the functions $g$ and $g_-$  coincide, which proves (\ref{eq0223-3}).

It remains to prove the asymptotic statements for $g$.
First note that
$$
g(\eta)\le f_1(\eta)\le \IE\, \vphi(\eta^{-1/d} \,\|X\|) \le \vphi(\eta^{-1/d} \sup_{x\in[0,1)^d} \|x\|) \ \longrightarrow \ 0
$$
as $\eta\to\infty$.
On the other hand, one has for $\eta>0$, $\eps>0$ and $N\iN$,
\begin{align*}
f_N(\eta) &=\IE \,\vphi\Bigl( (N/\eta)^{1/d}\, \|X-\hat X^{(N)}\| \Bigr)\\
&\ge \IE \Bigl[ 1_{\{\|X-\hat X^{(N)}\|\ge \eps/N^{1/d}\}}\,\vphi\Bigl( (N/\eta)^{1/d}\, \|X-\hat X^{(N)}\| \Bigr)\Bigr]\\
&\ge (1-N \,\lam^d(B(0,\eps/N^{1/d})))\,\vphi(\eps/\eta^{1/d})\\
&= (1-\lam^d (B(0,\eps))) \,\vphi(\eps/\eta^{1/d}),
\end{align*}
so that
$$
g(\eta)\ge (1-\lam^d (B(0,\eps))) \,\vphi(\eps/\eta^{1/d})\
\underset{\eta\dto0}{\longrightarrow} \ (1-\lam^d (B(0,\eps))) \,\sup_{t\ge0} \vphi(t).
$$
Since $g(\eta)\le \sup_{t\ge0} \vphi(t)$ and $\eps>0$ was arbitrary, the assertion follows.
\end{proof}

\section{The upper bound (1st step)}\label{sec3}

In this section, we consider an original $X$ with law $\mu\ll \lam^d$.
 Moreover, we assume that $\mu$ is compactly supported and fix $l>0$ large enough so that $\mu(C)=1$ for $C=[-l,l)^d$.

Based on a given integrable function $\xi:\IR^d\to [0,\infty)$ (\emph{point density}), we define codebooks and control their efficiency.

\begin{propo}\label{le030305a}
There exist codebooks $\cC(N)$, $N\ge 1$,  such that $\lim_{N\to\infty} \frac1N\,|\cC(N)|= \|\xi\|_{L^1(\IR^d)}$ and
$$
\limsup_{N\to\infty} \IE\,\vphi(N^{1/d}\, d(X,\cC(N)))\le \int g(\xi(x))\,d\mu(x).
$$
\end{propo}

\begin{proof}
It suffices to prove the assertion for functions $\xi$ that are uniformly bounded away from $0$ on $C$. If this is not the case, one can consider $\bar \xi=\xi+\eps\,1_{C}$ for some $\eps>0$. Then the statement says that there exist codebooks $\cC^\eps(N)$, $N\ge1$, with $|\cC^\eps(N)|\sim N(\|\xi\|_{L^1(\IR^d)}+\eps\,\lam^d(C))$ satisfying
$$
\limsup_{N\to\infty} \IE\,\vphi(N^{1/d}\, d(X,\cC^\eps(N)))\le \int g(\bar\xi(x))\,d\mu(x)\le \int g(\xi(x))\,d\mu(x),
$$
and a diagonalization argument for $\eps\dto0$ proves the general assertion.

Fix $m\iN$, let $C_1=[0,l/2^m)^d$ and decompose $C$ into a finite disjoint union
$$
C=\bigcup_{i=1}^M C_i,
$$
where $M=2^{(m+1)d}$ and $C_2,\dots,C_M$ are translates of $C_1$.
For $i=1,\dots,M$, we denote by $X_i$ a uniformly distributed r.v.\ on $C_i$, and let
$\mu^m=\sum_{i=1}^{M} \mu(C_i)\,\cU(C_i)$. Moreover, we let
$$
h_m =\frac{d \mu^m}{d\lambda^{d}}=\sum_{i=1}^{M}\frac{
  \mu(C_i)}{\lambda^{d}(C_i)} \cdot 1_{C_i},
$$
and denote by $\nu$ the measure given by $\nu(A)=\int_{A}\xi\,d\lam^d$, $A\in\cB(\IR^d)$.

We introduce the codebooks of interest. For some fixed $\kappa>0$, let
$$
\tilde \cC(N)=(\kappa\,N^{-1/d}\IZ^d)\cap C, \qquad N\ge1,
$$
and let $\cC_i(N)$ denote codebooks of size $N_i=N_i(N)=N \,\nu(C_i)$ minimizing
$\IE\, \vphi(N^{1/d}\, d(X_i,\cC_i(N)))$. We consider the efficiency of the codebooks
$$
\cC(N)=\tilde \cC(N)\cup \bigcup_{i=1}^M \cC_i(N),\qquad N\ge1.
$$
First, note that
$$
|\tilde \cC(N)|\le \bigl(1+2\frac l\kappa N^{1/d}\bigr)^d \sim \Bigl(\frac{2l}{\kappa}\Bigr)^d N,\qquad N\to\infty,
$$
hence:
\begin{align}\label{eq0309-2}
|\cC(N)|\lesssim \Bigl(\|\xi\|_{L^1(\IR^d)}+ \Bigl(\frac {2 l}{\kappa}\Bigr)^d\Bigr) \,N,\qquad N\to\infty .
\end{align}

It remains to estimate the expectation $\IE\,\vphi(N^{1/d} d(X,\cC(N)))$ for large $N\ge1$. Observe that $d(x,\tilde \cC(N))\le \kappa \,N^{-1/d}\,\sup_{x\in[0,1)^d}\|x\|$ for all $x\in C$ so that
\begin{align}\begin{split}\label{eq0309-3}
\Bigl|\IE\,&\vphi\bigl(N^{1/d} \,d(X,\cC(N))\bigr) - \int \vphi\bigl(N^{1/d} \,d(x,\cC(N))\bigr)\,d\mu^m(x)\Bigr|\\
&= \Bigl|\int\vphi\bigl(N^{1/d} \,d(x,\cC(N))\bigr) (h(x)-h_m(x)) \,dx\Bigr|\le  \vphi(c \kappa) \,\|h-h_m\|_{L^1(\IR^d)},
\end{split}\end{align}
where $c=\sup_{x\in[0,1)^d}\|x\|$ is a universal constant.
Moreover,
\begin{align*}
\int \vphi(N^{1/d} \,d(x,\cC(N))\,d\mu^m(x) & = \sum_{i=1}^{M}
\mu(C_i)\, \IE\, \varphi \bigl(N^{1/d} \,d(X_i,\cC(N))\bigr) \\
& \le \sum_{i=1}^{M}
\mu(C_i)\, \IE\, \varphi \bigl(N^{1/d} \,d(X_i,\cC_i(N))\bigr).
\end{align*}
Now let $U$ denote a $\cU([0,1)^d)$-distributed r.v. Due to the
optimality assumption on the choice of $\cC_i(N)$, a translation and
scaling then yields
$$
\IE\, \varphi \bigl(N^{1/d} \,d(X_i,\cC_i(N))\bigr) =\IE \, \varphi \Bigl(N^{1/d} \frac l{2^{m+1}}\,\|U-\hat U^{(N_i)}\| \Bigr),
$$
where $\hat U^{(N_i)}$ denotes a reconstruction minimizing the latter
expectation among all r.v.\ with a range of size $N_i$. Next, rewriting the previous expectation as
$$
\IE \, \varphi \Bigl(N^{1/d} \frac l{2^{m+1}}\,\|U-\hat U^{(N_i)}\| \Bigr)=f_{N_i}({\nu(C_i)/\lam^d(C_i)}),
$$
it follows that
$$
\int \vphi\bigl(N^{1/d} \,d(x,\cC(N))\bigr)\,d\mu^m(x) \le \sum_{i=1}^{M}
\mu(C_i)\, f_{N_i}({\nu(C_i)/\lam^d(C_i)}).
$$
As $N\to\infty$, every $N_i$, $i=1,\dots,M$, converges to $\infty$, and one has
$$
\sum_{i=1}^{M}
\mu(C_i)\, f_{N_i}({\nu(C_i)/\lam^d(C_i)})
 \ \longrightarrow \  \sum_{i=1}^{M}
\mu(C_i)\, g\Bigl(\frac{\nu(C_i)}{\lam^d(C_i)}\Bigr)= \int g(\xi_m(x))\,d\mu(x),
$$
where $\xi_m= \sum_{i=1}^M \frac{\nu(C_i)}{\lam^d(C_i)}\cdot 1_{C_i}$.
Putting everything together yields
$$
\limsup_{N\to\infty} \IE\,\vphi\bigl(N^{1/d} \,d(X,\cC(N))\bigr) \le  \int g(\xi_m(x))\,d\mu(x)+\kappa \,\|h-h_m\|_{L^1(\IR^d)}\,\sup_{x\in[0,1)^d}\|x\|.
$$
The function $\xi_m$ converges to $\xi$ as $m\to\infty$ in $\lam^d$-a.a.\ points $x$ (see  \citeasnoun{Cohn80}, Theorem 6.2.3).
Recall that by construction $\xi_m$ is bounded from below on $C$, and hence dominated convergence gives
$$
\lim_{m\to\infty} \int g(\xi_m(x))\,d\mu(x)=\int g(\xi(x))\,d\mu(x).
$$
Analogously, $h_m$ converges to $h$ in $\lam^d$-a.a.\ points $x$ and due to Scheff\'e's theorem (see \citeasnoun{Bil79}, Theorem 16.11) $h_m$ converges to $h$ in $L^1(\IR^d)$ as $m\to\infty$.

For arbitrary $\eps>0$,  we can choose $\kappa>0$ sufficiently large to ensure that the size of $\cC(N)$  (see (\ref{eq0309-2})) satisfies
$$
|\cC(N)|\lesssim (1+\eps) \,\|\xi\|_{L^1(\IR^d)}\,N .
$$
Finally, it remains to pick $m\iN$ sufficiently large so that
$$
\limsup_{N\to\infty} \IE\,\vphi\bigl(N^{1/d} \,d(X,\cC(N))\bigr) \le   \int g(\xi(x))\,d\mu(x) + \eps,
$$
and the general statement then follows from a diagonalization argument.
\end{proof}

\section{The lower bound}\label{sec4}

From now on, let $X$ be an arbitrary random vector on $\IR^d$ with law $\mu$.
In this section, we change our viewpoint: for an index set  $\II\subset[1,\infty)$ with $\sup\II=\infty$, we consider arbitrary finite codebooks $\cC(N)$, $N\in\II$, and ask for asymptotic lower bounds of
$$
\IE \,\vphi(N^{1/d} \,d(X,\cC(N)))
$$
as $N\to\infty$.
Our computations are based on the assumption that the empirical measures
$$
\nu^N= \frac1N\sum_{\hat x\in\cC(N)} \delta_{\hat x},\qquad N\in\II,
$$
associated to $\cC(N)$ converge vaguely to some locally finite measure $\nu$ on $\IR^d$.

\begin{propo}\label{theo0318-1}
Letting $\nu_c$ denote the absolutely continuous part of $\nu$, one has
$$
\liminf_{N\to\infty} \IE \,\vphi\bigl(N^{1/d}\,d(X,\cC(N))\bigr)\ge \int g\bigl(\frac{d\nu_c}{d\lam^d}\bigr)\,d\mu_c(x).
$$
\end{propo}

\begin{proof}
It suffices to prove that for an arbitrary $l>0$:
$$
\liminf_{N\to\infty} \IE \,\vphi\bigl(N^{1/d}\,d(X,\cC(N))\bigr)\ge \int_{[- l, l)^d} g\bigl(\frac{d\nu_c}{d\lam^d}\bigr)\,d\mu_c(x).
$$
Indeed, the  assertion then follows immediately by monotone convergence.
 For a given $m\iN$, just as in the proof of the upper bound, we
 decompose the set $C=[- l, l)^d$ into a disjoint union $
C=\bigcup_{i=1}^M C_i,
$
where $M=2^{m+1}$, $C_1=[0,l/2^m)^d$, and $C_2,\dots,C_M$ are translates of $C_1$.
Again we consider the measure $\mu^m=\sum_{i=1}^M \mu_c(C_i)\,\cU(C_i)$ and the density
$$
h_m =\sum_{i=1}^M \frac{\mu_c(C_i)}{\lam^d(C_i)} \cdot 1_{C_i}.
$$
Analogously, we let $\xi_m=\sum_{i=1}^M \frac{\nu(\bar C_i)}{\lam^d(C_i)} \cdot 1_{C_i}$.
For some fixed $\kappa>0$, we extend the codebooks $\cC(N)$ to
$$\cC^{(1)}(N)= \cC(N)\cup \bigl( (\kappa N^{-1/d} \IZ^d)\cap C\bigr).$$
Then, just as in (\ref{eq0309-3}), one has
\begin{align}\label{eq0318-1}\begin{split}
\Bigl|\int_C \vphi(N^{1/d}\,d(x,\cC^{(1)}(N)))  \, d\mu_c(x) -\int_C \vphi(N^{1/d}\,d(x,\cC^{(1)}(N)))\, d\mu^m(x)\Bigr|\le \vphi(c\,\kappa) \,\|h-h_m\|_{L^1(C)},
\end{split}
\end{align}
where $c=\sup_{x\in[0,1)^d}\|x\|$.

Next, we  control the approximation efficiency of $\cC^{(1)}(N)$ for the measure $\mu^m$.
We fix $\eps\in(0,l/2^{m+1})$ and  consider for $i=1,\dots,M$, the closed cubes
$$
C_i^\eps= \{x\iR^d: d_2(x,(C_i)^c)\ge\eps\}\subset C_i.
$$
Here $d_2$ denotes the standard Euclidean metric on $\IR^d$.
Now observe that there exists a finite set $\cK=\cK(\eps)\subset \IR^d$ such that for  $x\in C_i^\eps$, $i=1,\dots,M$,
\begin{align}\label{eq0317-2}
d(x,\cK\cap C_i)\le d(x,(C_i)^c).
\end{align}
We extend the codebooks $\cC^{(1)}(N)$ to  $\cC^{(2)}(N)=\cC^{(1)}(N)\cup\cK$ and let $\cC_i(N)=\cC^{(2)}(N)\cap C_i^{\eps}$ for $N\in\II$ and $i=1,\dots,M$.
Note that property (\ref{eq0317-2}) guarantees that any point $x$ in an arbitrary  cube $C_i^\eps$ has as best $\cC^{(2)}(N)$-approximant an element in $\cC_i(N)$. Moreover, none of the codebooks $\cC_i(N)$ is empty, i.e.\ the number $N_i=N_i(N)$ defined as  $N_i=|\cC_i(N)|$, is greater or equal to $1$.
Consequently, letting $X_i$ denote $\cU(C_i^\eps)$-distributed r.v.'s, one obtains
\begin{align}\begin{split}\label{eq0318-2}
\int \vphi(N^{1/d} d(x,\cC^{(2)}(N)))\,d\mu^m(x) & \ge \int_{\bigcup_{i=1}^M C_i^\eps} \vphi(N^{1/d} d(x,\cC^{(2)}(N)))\,d\mu^m(x)\\
&=\sum_{i=1}^M \mu^m(C_i^\eps) \,\IE \, \vphi(N^{1/d} d(X_i,\cC_i(N)))  .
\end{split}\end{align}
Let  $U$ be a $\cU([0,1)^d$-distributed r.v., and
fix an arbitrary $i\in\{1,\dots,M\}$. Note that the cube $C_i^\eps$
has side length $2^{-m} l-2\eps$, so that a shifting and rescaling yields
\begin{align*}
\IE \, \vphi(N^{1/d} d(X_i,\cC_i(N)))&\ge \IE \, \vphi(N^{1/d}\, (2^{-m}l-2\eps)\,\|U-\hat U^{(N_i)}\|)\\
&=\IE \, \vphi((N \lam^d(C_i^\eps))^{1/d}\,\|U-\hat U^{(N_i)}\|),
\end{align*}
where $\hat U^{(N_i)}$ denotes an optimal approximation satisfying the range constraint $|\range(\hat U^{(N_i)})|\le N_i$.
With $f_N(\eta)$ as in (\ref{eq0317-1}), we arrive at
$$
\IE \, \vphi(N^{1/d} d(X_i,\cC_i(N)))\ge f_{N_i}({N_i/(N\,\lam^d(C_i^\eps))}).
$$

We need to control the quantity $N_i/N$ for $N$ large.
Recall that $\cC^{(2)}(N)$ is the union of the sets $\cC(N)$, $\cK$ and $(\kappa N^{-1/d} \IZ^d)\cap C$, and the vague convergence of $\nu^N$ to $\nu$  implies that
$$
\limsup_{N\to\infty}\frac{|\cC(N)\cap C_i|}{N} \le \nu(\bar C_i).
$$
Moreover, the set $(\kappa N^{-1/d} \IZ^d)\cap C_i$ contains at most
$
\bigl(\frac{l\,2^{-m}}{\kappa N^{-1/d}}+1\bigr)^d
$
elements, so that
$$
\limsup_{N\to\infty}\frac{N_i}{N} \le \nu(\bar C_i)+\frac{\lam^d(C_i)}{\kappa^d }.
$$
Consequently, Theorem \ref{le0223-1} implies that
$$
\IE\, \vphi(N^{1/d} d(X_i,\cC_i(N)))\gtrsim g\Bigl(\Bigl(\nu(\bar C_i)+ \frac{\lam^d(C_i)}{\kappa^d }\Bigr)\Big/\lam^d(C_i^\eps)\Bigr).
$$
Combining this estimate with (\ref{eq0318-1}) and (\ref{eq0318-2}) yields
\begin{align*}
\int_C \vphi(N^{1/d}\,d(x,\cC))  \, d\mu_c(x) & \ge \int_C \vphi(N^{1/d}\,d(x,\cC^{(1)}))\, d\mu^m(x) - \vphi(c\,\kappa) \,\|h-h_m\|_{L^1(C)}\\
&\gtrsim \sum_{i=1}^M \mu^m(C_i^\eps) \, g\Bigl(\Bigl(\nu(\bar C_i)+ \frac{\lam^d(C_i)}{\kappa^d }\Bigr)\Big/\lam^d(C_i^\eps)\Bigr) \\
& \hspace{4.2cm} -\vphi(c\,\kappa) \,\|h-h_m\|_{L^1(C)}
\end{align*}
as $N\to\infty$. Since $\eps>0$ can be chosen arbitrarily small, it follows that
\begin{align*}
\int_C \vphi(N^{1/d} d(x,\cC^{(2)}))  \, d\mu_c(x) & \gtrsim \sum_{i=1}^M \mu_c(C_i) \, g\Bigl(\frac{\nu(C_i)}{\lam^d(C_i)} + \frac{1}{\kappa^d }\Bigr)-\vphi(c\,\kappa) \,\|h-h_m\|_{L^1(C)}\\
&= \int_C g\Bigl(\xi_m(x)+\frac1{\kappa^d}\Bigr)\,d\mu_c(x) - \vphi(c\,\kappa) \,\|h-h_m\|_{L^1(C)}
\end{align*}
As $m\to\infty$, the densities $\xi_m$ converge pointwise to
$\xi=\frac{d\nu_c}{d\lam^d}$ for $\lam^d$-a.a.\ $x$, and $h_m$ converges to $h$ in $L^1(C)$.
Consequently, Fatou's Lemma implies that
$$
\liminf_{N\to\infty} \int_C \vphi(N^{1/d} d(x,\cC^{(2)}))  \, d\mu_c  \ge \int_C g\Bigl(\xi(x)+\frac1{\kappa^d}\Bigr)\,d\mu_c(x).
$$
Finally, observing that $\kappa>0$ was arbitrary and applying monotone convergence yields the general result.
\end{proof}

The above proposition enables us to give a partial proof of Theorem \ref{th0519-1}.
For the remainder of this section, let $I$ be given by (\ref{eq0519-1}), assume that $I\in(0,\infty)$, and denote
by $\cM$  the set of measures associated with the minimizers of the point allocation problem as defined in~(\ref{eq0525-1}). So far we have not proved that $\cM$ is non-empty.

\begin{propo}\label{prop4.2}
Suppose that the codebooks $\cC(N)$, $N\iN$, are of size $N$ and satisfy
\begin{align}\label{eq0323-1}
\limsup_{N\to\infty} N^{1/d}\, \|d(X,\cC(N))\|_\vphi\le I^{1/d},
\end{align}
and consider the associated empirical measures
$$
\nu^N=\frac1N\sum_{\hat x\in\cC(N)} \delta_{\hat x},\qquad N\iN.
$$
Then $(\nu^N)_{N\iN}$ is a tight sequence of probability measures
 and any accumulation
 point of $(\nu^N)$ lies in $\cM$ (in the weak topology). 
\end{propo}

\begin{proof}
Fix an arbitrary vaguely convergent subsequence $(\nu^{N})_{N\in\II}$
of $(\nu^N)_{N\in \IN}$ and denote by $\nu$ its limiting measure. Let $\eps>0$.
As long as the Orlicz norm $\|d(X,\cC(N))\|_\vphi$ is finite, one has in general
$$
\IE \vphi\bigl(\frac {d(X,\cC(N))}{\|d(X,\cC(N))\|_\vphi}\bigr)\le 1.
$$
Note that (\ref{eq0323-1}) implies that for all sufficiently large $N\iN$
$$
\|d(X,\cC(N))\|_\vphi\le ((1+\eps) I/N)^{1/d}
$$
so that
$$
\limsup_{N\to\infty} \IE\, \vphi((N/(1+\eps)I)^{1/d}\, d(X,\cC(N)))\le 1.
$$
We consider the codebooks $\tilde \cC(\tilde N)=\cC((1+\eps) \tilde N I)$ for
$$\tilde N \in \tilde\II:=\{N/((1+\eps)I):N\in\II\}.
$$
Then
\begin{align}\label{eq0323-2}
\limsup_{\tilde N\to\infty} \IE\,\vphi(\tilde N^{1/d}\, d(X,\tilde \cC(\tilde N)))\le 1.
\end{align}
On the other hand, the empirical measures
$$
\tilde\nu^{\tilde N}:=\frac1{\tilde N} \sum_{\tilde x\in\tilde\cC(\tilde N)} \delta_{\hat x} = (1+\eps)  I
\nu^{(1+\eps) \tilde N I}
$$
converge vaguely to $(1+\eps) I \nu$ so that by Theorem \ref{theo0318-1},
$$
\liminf_{\tilde N\to\infty} \IE\,\vphi(\tilde N^{1/d}\, d(X,\tilde \cC(\tilde N))) \ge \int g((1+\eps) I\,\xi(x))\,d\mu_c(x),
$$
where $\xi=\frac{d\nu_c}{d\lam^d}$. Combining this  with
(\ref{eq0323-2}), and noticing that $\eps>0$ is arbitrary, one obtains
$$
\int g( I\,\xi(x))\,d\mu_c(x)\le 1.
$$
Consequently, the point allocation $\tilde \xi(x)=I\,\xi(x)$ solves the allocation problem:
\begin{align}\label{eq0526-1}
\int g( \tilde\xi(x))\,d\mu_c(x)\le 1\text{ and }\int \tilde \xi\,d\lam^d\le I.
\end{align}
Due to the definition of $I$, the right inequality is actually an equality.

Assume now that $\int g( \tilde\xi(x))\,d\mu_c(x)< 1$, and fix
$\delta>0$ (small) so that the set $A:=\{x\iR^d:\tilde\xi(x)\ge \delta\}$ has positive Lebesgue measure. Since $g$ restricted to $[\delta/2,\infty)$ is Lipschitz continuous, we can lower the density  $\tilde \xi$ on $A$ in such a way that the point allocation constraint remains valid, thus contradicting the optimality of $I$.
Consequently, the inequalities in (\ref{eq0526-1}) are even
equalities, and we immediately obtain that $\nu_c\in\cM$. Since $\nu_c$ has mass $1$, we also have that $\nu=\nu_c\in\cM$. Moreover, $(\nu^{N})_{N\in\II}$ converges to $\nu$ in the weak topology.
We finish the proof by noticing that the sequence $(\nu^N)$ is tight,
since it has no vaguely convergent subsequence loosing some of its mass.
\end{proof}

\section{The point allocation problem}\label{sec5}

We decompose the original measure $\mu$ into its absolutely continuous part $\mu_c=h\,d\lam^d$ and singular component $\mu_s$. The singular component will have no influence on the asymptotics of the quantization error.

In this section we use standard methods for convex optimization problems to treat the point allocation problem, i.e.\  the
minimization of
\begin{align}\label{eq0304-1}
\int_{\IR^d} \xi(x)\,dx
\end{align}
over all positive integrable functions $\xi:\IR^d\to[0,\infty)$
satisfying
\begin{align}\label{eq0307-3}
\int_{\IR^d} g(\xi(x))\,d\mu_c(x)\le 1.
\end{align}
A minimizer $\xi$ will be called \emph{optimal point density}.

We shall use the convex conjugate of $g$, i.e.
$$
g^*(a)=\sup_{\eta\ge0}[a \eta -g(\eta)],\qquad a\le0,
$$
and the concave function $\bar g:[0,\infty)\to [0,\infty), a\mapsto -g^*(-a)$. Alternatively, one can define $\bar g$ as $
\bar g(a)=\inf_{\eta\ge0} [a\eta+g(\eta)].
$

The function $\bar g$ is continuous and satisfies $\bar g(0)=\inf_{\eta\ge0} g(\eta)=0$. The right continuity in $0$ is a consequence of the lower semicontinuity of $g^*$.
Moreover, since $g$ is lower semicontinuous, one has
\begin{align}\label{eq0316-2}
g(\eta)=\sup_{a\le0} [a\eta-g^*(a)]= \sup_{a\ge0}[ \bar g(a)-a\eta], \qquad \eta\ge0.
\end{align}

\begin{theo}\label{theo0315-1}
\begin{enumerate}\item
The minimal value $I$ satisfies the dual formula
\begin{align}\label{eq0307-1}
I = \sup_{\kappa>0} \frac1\kappa\Bigl(\int \bar g\bigl(\frac{\kappa}{h(x)}\bigr)\,d\mu_c(x)-1 \Bigr).
\end{align}
\item
The optimization problem has an integrable solution iff the integral
\begin{align}\label{eq0316-1}
\int \bar g\bigl(\frac{\kappa}{h(x)}\bigr)\,d\mu_c(x)
\end{align}
is finite for some $\kappa>0$. In such a case there exists an optimal point density $\xi$.

\item Suppose that (\ref{eq0316-1}) is finite and that
$$\mu_c(\IR^d) \sup_{t\ge0} \vphi(t)>1.$$
Then $I>0$ and  there exists an optimal point density. Moreover, all optimal point densities $\xi$ satisfy
\begin{align}\label{eq0526-3}\int g(\xi(x))\,d\mu_c(x)=1 \ \ \text{ and } \ \
\bar g'_+\bigl(\frac\kappa{h(x)} \bigr) \le \xi(x) \le \bar g'_-\bigl(\frac\kappa{h(x)} \bigr) \text{ for a.e.\ }x\iR^d,
\end{align}
where $\kappa$ is a maximizer of the right hand side of
(\ref{eq0307-1}). (Here we  make use of the convention that $\bar g'_+(\infty)=\bar g'_-(\infty)=0$).
In particular, the supremum in the dual formula is attained. 

\item If $\mu_c(\IR^d) \sup_{t\ge0} \vphi(t)\le1$, then $I=0$ and $\xi=0$ is an optimal point density.
\end{enumerate}
\end{theo}

\begin{remark}\label{re0406-1}
Assume that $I\in(0,\infty)$ and that $g$ is strictly convex. A standard result from convex analysis (see  \citeasnoun{Rock72}, Theorem 26.3) states that the strict convexity is equivalent to differentiability of $\bar g$. Hence, one obtains a one parameter family of candidates as optimal point densities. Moreover,
\begin{align}\label{eq1006-1}
\bar g'(a)= \inf\{b>0:  -g'_+(b)\le a\}.
\end{align}
Additionally, the strict convexity implies almost everywhere uniqueness, since  for two optimal solutions $\xi_1,\xi_2$ that were not almost everywhere identical the combination $\bar\xi=\frac12 (\xi_1+\xi_2)$ would satisfy $\int g(\bar\xi(x))\,dx<1$ and by continuity of $g$ it is straight forward to construct an admissible  density with smaller $L^1$-norm.

If  $g$ is additionally differentiable, the one parameter family of candidates is given via
$$
\xi^\kappa(x)=(-g')^{-1}\bigl(\frac\kappa{h(x)}\bigr)
$$
with the convention that $(-g')^{-1}(\infty)=0$.
\end{remark}

\begin{proof}[ of Theorem \ref{theo0315-1}] By the concavity of $\bar g$, the integral (\ref{eq0316-1}) is either finite  or infinite for all $\kappa>0$.
We start with proving the
  ``$\ge$'' inequality in the dual formula.
Note that by definition of $\bar g$, $
a\,b\ge \bar g(a)-g(b)$ for $a,b\ge 0$.
Therefore, for $\kappa>0$ and $\xi$ satisfying (\ref{eq0307-3}), it is true that
\begin{align}\begin{split}\label{eq0308-2}
\int \xi(x)\,dx&\ge\int_{\{h>0\}} \xi(x)\,dx= \frac1\kappa \int \frac\kappa{h(x)}\,\xi(x) \,d\mu_c(x) \\
&\ge \frac1\kappa\Bigl( \int \bar g\Bigl(\frac\kappa{h(x)}\Bigr) \,d\mu_c(x)-\int g(\xi(x))\,d\mu_c(x)\Bigr)
\ge \frac1\kappa\Bigl( \int \bar g\Bigl(\frac\kappa{h(x)}\Bigr)  \,d\mu_c(x)-1\Bigr).
\end{split}\end{align}

In order to have equalities in the above estimates, we need to find a density $\xi$ and $\kappa>0$ such that
\begin{eqnarray}\label{cond1}
&\xi(x)=0 \text{ for }\lam^d \text{ a.a. }x\in\{h=0\},\\
\label{cond2}
&\bar g\Bigl(\frac\kappa{h(x)}\Bigr)-g(\xi(x))=\frac\kappa{h(x)}\,\xi(x),\qquad\text{for }\mu_c\text{ a.a. }x
\end{eqnarray}
and
\begin{align}\label{cond3}
\int g(\xi(x))\,d\mu_c(x)=1.
\end{align}

In the case $\mu_c(\IR^d) \sup_{t\ge0}\vphi(t)\le 1$, it is easily seen that $\xi=0$ is an optimal point density so that $I=0$ which proves assertion 4. Moreover, the term on the right hand side of (\ref{eq0307-1}) tends to $0$ when letting $\kappa\to\infty$ so that the dual formula is valid in that case. Moreover, if~(\ref{eq0317-1}) is infinite for one $\kappa>0$, then there is no integrable nonnegative function $\xi$ satisfying~(\ref{eq0307-3}) and the dual formula is valid as well.

From now on, we assume that $\mu_c(\IR^d) \sup_{t\ge0}\vphi(t)>1$ and that~(\ref{eq0317-1}) is finite for any $\kappa>0$.
Next, we derive a density $\xi$ satisfying the three abovementioned conditions. Then estimate (\ref{eq0308-2}) implies optimality for this choice of $\xi$ which proves assertion 1.

First we examine the second condition.
Consider $a_0,b_0\ge0$ with
\begin{align}\label{eq0308-1}
 \bar g'_+(a_0)\le b_0 \le \bar g'_-(a_0).
\end{align}
Due to the concavity of $\bar g$ it holds that $\bar g(a)\le \bar g(a_0) +b_0 (a-a_0)$ for all $a\ge0$, and we obtain with (\ref{eq0316-2}),
$$
g(b_0)= \sup_{a\ge0}[ \bar g(a)-  b_0 \,a]\le  \sup_{a\ge 0}[ \bar g(a_0)+b_0(a-a_0) -b_0\,a]= \bar g(a_0)-b_0a_0.
$$
Consequently, condition (\ref{eq0308-1}) implies that
\begin{align}\label{eq0526-2}
g(b_0)= \bar g(a_0)-b_0\,a_0.
\end{align}
Conversely, it is easy to see that any pair $(a_0,b_0)$ of nonnegative reals satisfying~(\ref{eq0526-2}) also satisfy~(\ref{eq0308-1}).

For  $\kappa>0$, we consider the point densities
$$
\xi^\kappa_+(x)= \bar g'_+\Bigl(\frac\kappa{h(x)}\Bigr), \qquad x\iR^d,
$$
and
$$
\xi^\kappa_-(x)= \bar g'_-\Bigl(\frac\kappa{h(x)}\Bigr), \qquad x\iR^d
$$
with the convention $\bar g'_+(\infty)=\bar g'_-(\infty)=0$ so that, in particular,
   $\xi_-^\kappa(x)=\xi_+^\kappa(x)=0$ for $x\in \{h=0\}$.
Furthermore, for any $x\in\{h>0\}$, condition (\ref{eq0308-1}) is satisfied for $a_0=\kappa/h(x)$ and for all $b_0\in [\xi_+^\kappa(x), \xi_-^\kappa(x)]$. Therefore, every convex combination $\bar\xi=\alpha\,\xi_+^\kappa+(1-\alpha)\, \xi_-^\kappa$ satisfies
\begin{align} \label{eq0308-4}
\bar g\Bigl(\frac\kappa{h(x)}\Bigr)-g(\bar\xi(x))=\frac\kappa{h(x)}\,\bar\xi(x) \ \text{ on } \ \{h>0\}.
\end{align}
In analogy to (\ref{eq0308-2}) we obtain that
$$
\int \bar \xi(x)\,dx= \frac1\kappa \Bigl( \int \bar g\Bigl(\frac\kappa{h(x)}\Bigr) \,d\mu_c(x)-\int g(\bar \xi(x))\,d\mu_c(x)\Bigr).
$$
In particular, all three integrals are finite. It remains to find an appropriate $\kappa>0$ and a  convex combination $\bar\xi$ as above with
$$
\int g(\bar \xi(x))\, d\mu_c(x)=1.
$$

We need to compute the asymptotic behavior of $g(\bar g'_-(a))$ as $a\to0$ and $a\to\infty$.
Due to equation (\ref{eq0526-2}), one has for all $a>0$,
$$
g(\bar g'_-(a))=\bar g(a)-\bar g'_-(a)\, a\le \bar g(a),
$$
and we obtain that $\lim_{a\dto0} g(\bar g'_-(a))=0$.
On the other hand, for any $b>0$,
$$
g(b)=\sup_{a\ge0}[\bar g(a)-a\,b] \ge \limsup_{a\to\infty} a\,\Bigl(\frac{\bar g(a)}{a}-b\Bigr),
$$
and, since $g(b)$ is finite, it follows that $\lim_{a\to\infty} \bar g(a)/a=0$. Thus the concavity of $\bar g$ implies that $\lim_{a\to\infty} \bar g'_-(a)=0$ and we arrive at
$$
\lim_{a\to\infty} g(\bar g_-'(a))=g(0).
$$

The above asymptotics imply with monotone convergence, that
$$
\lim_{\kappa\to0} \int g(\xi^\kappa_-(x))\, d\mu_c(x)= 0,
$$
and
$$
\lim_{\kappa\to\infty} \int g(\xi^\kappa_-(x))\, d\mu_c(x)= g(0)\,\mu_c(\IR^d)>1.
$$
Now let
$$
\kappa_0=\sup\Bigl\{ \kappa>0: \int g(\xi_-^\kappa(x))\,d\mu_c(x)\le 1\Bigr\}.
$$
As we have seen above the set is non-empty and bounded so that $\kappa_0\in(0,\infty)$. Moreover, since for any $x\iR^d$,  $\kappa\mapsto \xi^\kappa_-(x)$ is decreasing and left continuous, it follows by monotone convergence that
\begin{align}\label{eq0308-6}
\int g(\xi_-^{\kappa_0}(x))\,d\mu_c(x)= \lim_{\kappa\uparrow\kappa_0} \int g(\xi_-^{\kappa}(x))\,d\mu_c(x)\le 1.
\end{align}
Similarly the inequality  $\xi_+^\kappa (x)\le \xi_-^\kappa(x)$ implies that
\begin{align}\begin{split}\label{eq0308-5}
\int g(\xi_+^{\kappa_0}(x))\,d\mu_c(x)&=\lim_{\kappa\dto\kappa_0} \int g(\xi_+^{\kappa}(x))\,d\mu_c(x)\\
&\ge \lim_{\kappa\dto\kappa_0}\int  g(\xi_-^{\kappa}(x))\,d\mu_c(x)\ge1.
\end{split}
\end{align}
Hence, inequalities (\ref{eq0308-6}) and (\ref{eq0308-5}) imply the existence of a convex combination $\bar\xi=\alpha\, \xi_+^\kappa+(1-\alpha)\, \xi_-^\kappa$ with
$$
\int g(\bar \xi(x))\,d\mu_c(x)=1.
$$
This function $\bar \xi$ solves
$$
\int \bar\xi(x)\,dx= \frac1{\kappa_0}\Bigl( \int \bar g\Bigl(\frac{\kappa_0}{h(x)}\Bigr)  \,d\mu_c(x)-1\Bigr)=\sup_{\kappa>0}\frac1{\kappa}\Bigl( \int \bar g\Bigl(\frac{\kappa}{h(x)}\Bigr)  \,d\mu_c(x)-1\Bigr),
$$
and we proved assertion 1. Moreover, we observe that for any $\xi$ which is not of the form~(\ref{eq0526-3}) there is a strict inequality in at least one of the estimates in (\ref{eq0308-2}).

\end{proof}

\section{The singular case}
In this section, we consider an original $X$ with law $\mu\bot\lam^d$ or, equivalently,  $\mu_c=0$. Moreover, we again assume that $\mu$ is compactly supported.

\begin{propo}\label{le0503-1}
There exist codebooks $\cC(N)$, $N\geq 1$,  with $\lim_{N\to\infty} \frac1N |\cC(N)|=0$ such that
$$
\lim_{N\to\infty} \IE \,\varphi(N^{1/d}\, d(X,\cC(N)))
=0.
$$
\end{propo}

\begin{proof} We fix $\eps>0$.
As the measure $\mu$ is singular with respect to the Lebesgue measure
$\lambda^d$, there exists an open set $A\subset
\IR^d$ with $\mu(A)=1$ and $\lambda^d(A)\leq \varepsilon$.
Due to Lemma 1.4.2 in \citeasnoun{Cohn80}, one can represent the open set $A$ as a
countable disjoint union of half-open cubes $(C_i)_{i\iN}$ in $\bigcup_{m=1}^\infty \cB_m$, where
$$
\cB_m=\bigl\{ [i_1 2^{-m},(i_1+1) 2^{-m})\times\dots\times [i_d 2^{-m},(i_d+1) 2^{-m}): i_1,\dots,i_d\iZ \bigr\}  \subset\IR^d
$$
for $m\iN$. Due to monotone convergence, we obtain that there exists $M\iN$ with
\begin{align}\label{eq0323-3}
\mu\Bigl(\bigcup_{i=1}^M C_i\Bigr)\ge \mu(A)-\eps = 1-\eps.
\end{align}
Set $C=\bigcup_{i=1}^M C_i$.
%

Let us introduce the codebooks; fixing $l>0$ such that $\supp(\mu)\subset [-l,l)^d$, the construction depends upon two parameters $\kappa_1,\kappa_2>0$:
$$
\cC(N)= \bigl((\kappa_1 \,N^{-1/d}\, \IZ^d)\cap C\bigr) \cup \bigl((\kappa_2 \,N^{-1/d}\, \IZ^d)\cap [-l,l)^d\bigr), \qquad N\ge1.
$$
We need to control the size of $\cC(N)$. For $i\in\{1,\dots,M\}$, let $m_i\iN$ denote the unique number with $C_i\in\cB_{m_i}$, and observe that
$$
|(\kappa_1 \,N^{-1/d}\, \IZ^d)\cap C_i| \le \Bigl(\frac{2^{-m_i}}{\kappa_1\,N^{-1/d}} +1\Bigr)^d \sim \lam^d(C_i) \,\kappa_1^{-d} \,N
$$
as $N\to\infty$.
Analogously,
$$
|(\kappa_2 \,N^{-1/d}\, \IZ^d)\cap [-l,l)^d| \le \Bigl(\frac{2l }{\kappa_2\,N^{-1/d}} +1\Bigr)^d \sim (2l)^d \,\kappa_2^{-d} \,N.
$$
Consequently,
\begin{align*}
|\cC(N)|&\le \sum_{i=1}^M |(\kappa_1 \,N^{-1/d}\, \IZ^d)\cap C_i|+ |(\kappa_2 \,N^{-1/d}\, \IZ^d)\cap [-l,l)^d|\\
&\lesssim \bigl(\lam^d(C) \,\kappa_1^{-d}+(2l)^d \,\kappa_2^{-d}\bigr)\,N
\le \bigl(\eps \,\kappa_1^{-d}+(2l)^d \,\kappa_2^{-d}\bigr)\,N.
\end{align*}

Next, we estimate the approximation error. Suppose that $N\ge 1$ is
sufficiently large so that  $C_i\cap \cC(N)\not =\emptyset$ for all $i=1,\dots,M$. Let $c=\sup_{x\in[0,1)^d}\|x\|$, and observe that for all $x\in C$, $
d(x,\cC(N))\le c\,\kappa_1 \,N^{-1/d}$.
Moreover, for any $x\in[-l,l)^d$,
$
d(x,\cC(N))\le c\,\kappa_2\,N^{-1/d}$.
Consequently,
\begin{align}\begin{split}\label{eq0323-4}
\IE \,\varphi( N^{1/d}\, d(X, \cC(N))) &\le \mu(C) \,\vphi(c\,\kappa_1)+(1-\mu(C)) \,\vphi(c\,\kappa_2)\\
& \le \vphi(c\,\kappa_1)+\eps \,\vphi(c\,\kappa_2).
\end{split}
\end{align}

Now, for $\delta>0$ arbitrary, pick $\kappa_1,\kappa_2>0$ satisfying
$
\vphi(c\,\kappa_1)\le \delta/2$  and $ (2l)^d/\kappa_2^d\le \delta/2$,
and choose $\eps>0$ so that
$
\eps \kappa_1^{-d}\le \delta/2 \ \text{ and } \ \eps\,\vphi(c\,\kappa_2)\le \delta/2$.
Then the corresponding codebooks $\cC(N)$ satisfy
$$
|\cC(N)|\lesssim \delta \,N \ \text{ and } \ \IE \,\varphi( N^{1/d}\, d(X, \cC(N))) \lesssim \delta,
$$
and the assertion of the proposition follows by a diagonalization argument.
\end{proof}

\section{Extension to the non-compact setting}\label{sec6}

In order to treat the non-compact quantization problem, we need to control the impact of realizations lying outside large cubes. For $L^p(\IP)$-norm distortions, \citeasnoun{Pie70} (see also \citeasnoun{GraLu00},  Lemma 6.6) discovered that the quantization error can be estimated against a higher moment $\tilde p>p$ of $\|X\|$.  His result can be easily extended to the inequality
$$
\delta(N|X,p)\le C\, \IE[ \|X\|^{\tilde p}]^{1/\tilde p} N^{-1/d},
$$
where $X$ is an arbitrary original in $\IR^d$, $N\iN$ and $C$ is a universal constant depending only on $E$, $p$ and $\tilde p$.
Pierce's proof is based on a random coding argument. In contrast to
his approach, we will use $\eps$-nets to establish a similar result.

The construction is based on several parameters. Let $\Psi:[0,\infty)\to [0,\infty)$ denote an increasing function, $(r_n)_{n\iN_0}$ an increasing sequence, and let $(\alpha_n)_{n\iN}$ be a positive decreasing and summable sequence.

\begin{lemma}\label{le0324-1}
Let $J\iN_0$ and denote by $X$ a $\left(B(0,r_{J})^c\cup\{0\}\right)$-valued r.v. For $N\geq0$ there exists a codebook $\cC(N)$ of size
$1+N \,\sum_{n=J}^\infty \alpha_n$ satisfying
\begin{align}\label{eq0309-1}
\IE \,\vphi\bigl(N^{1/d} \,d(X,\cC(N))\bigr)\le \IE[\Psi(\|X\|)] \sum_{n=J}^{\infty} \frac1{\Psi(r_n)} \,\vphi(c_E \, \alpha_{n}^{-1/d} \,r_{n+1}),
\end{align}
where $c_E$ is a finite constant depending on the norm $\|\cdot\|$ only.
\end{lemma}

\begin{proof}
First, observe that the  sum in estimate (\ref{eq0309-1}) diverges whenever $\liminf_{n\to\infty} r_n<\infty$. Thus we can assume without loss of generality that $\lim_{n\to\infty} r_n=\infty$.
Fix $J\iN_0$ and $N\iN$.
For $n\iN_0$, let $V_n=B(0,r_n)$ and $N_n=\alpha_n N$.
Moreover, we denote by $I\iN_0$ the smallest index $n$ with $N_n<1$, and let for $n\iN_0$ with $J\le n<I$, $\cC_n$ denote an optimal $\eps$-net for $B(0,r_{n+1})$ consisting of $N_n$ elements. As is well known there exists a constant $c_E$ only depending on the norm $\|\cdot\|$ such that $d(x,\cC_n)\leq c_E r_{n+1} N_n^{-1/d}$  for all $x\in B(0,r_{n+1})$.
Thus the codebook $\cC=\{0\}\cup\bigcup_{n=J}^{I-1}\cC_n $ contains at most $1+N\sum_{n=J}^\infty\alpha_n$ elements, and we have
\begin{align*}
\IE\,\vphi\bigl(N^{1/d} \,d(X,\cC) \bigr)& = \sum_{n=J}^\infty \IE\Bigl[1_{V_{n+1}\backslash V_{n}}(X)\,\vphi\bigl(N^{1/d} \,d(X,\cC)\bigr)\Bigr]\\
&\le\sum_{n=J}^{\infty}  \IP(X\not \in V_{n}) \,\vphi\bigl(c_E\, r_{n+1}  \, N^{1/d}/(1\vee N_{n})^{1/d}\bigr)\\
&\le \sum_{n=J}^{\infty}  \IP(X\not \in V_{n}) \,\vphi(c_E \,r_{n+1} \,\alpha_{n}^{-1/d} )\\
&\le \IE[\Psi(\|X\|)] \,\sum_{n=J}^{\infty} \frac1{\Psi(r_n)} \,\vphi(c_E \,r_{n+1}\, \alpha_{n}^{-1/d}).
\end{align*}
\end{proof}

\begin{defi}\label{defi0322-1}
We say that an increasing function $\Psi:[0,\infty)\to[0,\infty)$ satisfies the growth condition (G) for $\vphi$ iff there exists a decreasing summable  sequence $(\alpha_n)$ and an increasing sequence $(r_n)$ with
$$
\sum_{n=0}^\infty\frac1{\Psi(r_n)} \,\vphi(\alpha_{n}^{-1/d} \,r_{n+1})<\infty.
$$
\end{defi}

Suppose that $\Psi$ satisfies condition (G) for $\vphi$.
We shall see that the condition that $\IE \Psi(\|X\|)<\infty$ is
sufficient to conclude that the quantization problem for $X$ under the
Orlicz norm induced by $\vphi$ is of order $N^{-1/d}$. Moreover, quantization of
non-compact measures then  can  be approximated by the compact setting.

\begin{example}\label{exa0817-1}\begin{itemize}\item
Suppose that $\vphi(t)\leq c (1+t^p)$ for all $t\geq0$, where $c,p\iR_+$ are appropriate constants. Then, for any $\beta>\frac{p+d}d$, 
$$
\Psi(t)=t^p (\log_+ t)^\beta
$$
satisfies (G) for $\vphi$. For instance, one may choose $(r_n)_{n\iN_0}=(2^n)_{n\iN_0}$ and $(\alpha_n)_{n\iN_0}=( (n+2)^{-\gamma})_{n\iN_0}$ for a $\gamma\in(1,\frac dp(\beta -1))$.
\item
Suppose that $\vphi$ satisfies $\vphi(t)\le c\,\exp\{t^\kappa\}$ for all $t\ge0$, where $c,\kappa\iR_+$ are appropriate constants.
Then, for any $\tilde \kappa>\kappa$, the function
$$
\Psi(t)= \exp\{ t^{\tilde\kappa}\}
$$
satisfies (G) for $\vphi$, as can be verified easily for  $(\alpha_n)_{n\iN_0}=( (n+2)^{-2})_{n\iN_0}$ and $(r_n)_{n\iN_0}=((n+1)^s)_{n\iN_0}$ for $s>0$ with $s \tilde\kappa > (\frac2d+s)\kappa$.
\end{itemize}
\end{example}

\begin{remark}
The proof of the upper bound in Theorem \ref{theo0323-2} relies on the assumption that $\IE\Psi(\|X\|)<\infty$ for some $\Psi$ satisfying the growth condition (G). As we shall see below, this assumption can be replaced by the equivalent condition that  $X\in L^\Psi(\IP)$ for some $\Psi$ satisfying (G).
First, assume that $\IE \Psi(\|X\|)<\infty$ for some $\Psi$ satisfying (G). Then $\tilde\Psi=1_{[1,\infty)}\, \Psi$ satisfies (G), and since by monotone convergence
$$
\lim_{\kappa\to\infty} \IE \Psi(\|X\|/\kappa)=0,
$$
it follows that $X\in L^{\tilde\Psi}(\IP)$.
On the other hand, assuming that $X\in L^{\Psi}(\IP)$ for some $\Psi$
satisfying (G) implies the existence of a $\kappa>0$ for which
$$
\IE \Psi(\|X\|/\kappa)<\infty.
$$
Now, let $\tilde \Psi(t)=\Psi(t/\kappa)$, and denote by $(\alpha_n)$ and $(r_n)$ sequences as in Definition \ref{defi0322-1}. Then $\IE \tilde\Psi(\|X\|)<\infty$, and
$$
\sum_{n=0}^\infty\frac1{\tilde\Psi(\tilde r_n)} \,\vphi(\tilde \alpha_{n}^{-1/d} \,\tilde r_{n+1})=\sum_{n=0}^\infty\frac1{\Psi(r_n)} \,\vphi(\alpha_{n}^{-1/d} \, r_{n+1})<\infty
$$
for $\tilde \alpha_n=\kappa^d \alpha_n$ and $\tilde r_n=\kappa r_n$, $n\iN_0$.
\end{remark}

We now combine the quantization results for continuous, singular, and unbounded measures to finish the proof of the upper bound in Theorem \ref{theo0323-2}.

\begin{proof}[ of Theorem \ref{theo0323-2}]
Let $\Psi$ be a function satisfying (G). It is easy to see that there exist also a summable and decreasing sequence $(\alpha_n)$ and an increasing sequence $(r_n)$ such that
$$
\sum_{n=0}^\infty\frac1{\Psi(r_n)} \,\vphi(c_E\,\alpha_{n}^{-1/d} \,r_{n+1})<\infty,
$$
where $c_E$ is as in Lemma \ref{le0324-1}.
We denote by $\xi$ an optimal point density, so that $\xi$ satisfies
$$
\int_{\IR^d} g(\xi(x))\,d\mu_c(x)=1 \ \text{ and } \ \int_{\IR^d} \xi(x)\,dx=I
$$
or $\xi=0$ (in the case $I=0$).
We fix $\eps>0$ and let $\tilde \xi(x)=\xi(x)+\eps h(x)$, $x\iR^d$. This point density satisfies
$$
\int_{\IR^d} g(\tilde\xi(x))\,d\mu_c(x)<1
$$
since $g$ is strictly  decreasing on $\{\eta>0: g(\eta)>0\}$. Now fix $J\iN_0$ such that
\begin{align}\label{eq0503-1}
\int_{\IR^d} g(\tilde\xi(x))\,d\mu_c(x)+ \IE[\Psi(\|X\|)]\sum_{n=J}^\infty \frac1{\Psi(r_n)} \vphi(c_E\,\alpha_n^{-1/d} r_{n+1})<1
\end{align}
and $\sum_{n=J}^\infty \alpha_n <\eps$.
Next, decompose the measure $\mu$ into the sum $\mu=\tilde\mu_c+\tilde\mu_s+\mu_u$, where $\tilde\mu_c$ and $\tilde\mu_s$ are the absolutely continuous and singular part of $\mu$ restricted to $B(0,r_J)$, respectively, and $\mu_u$ contains the rest of the mass of $\mu$.

It remains to combine the former results.
Due to Proposition \ref{le030305a} there exist codebooks $\cC_1(N)$, $N\ge1$, with $\lim_{N\to\infty} \frac1N |\cC_1(N)|=\|\tilde \xi\|_{L^1(\IR^d)}$ and
$$
\limsup_{N\to\infty} \int \vphi(N^{1/d} d(x,\cC_1(N)))\,d\tilde\mu_c\le \int_{\IR^d} g(\tilde\xi(x))\,d\mu_c(x).
$$
Moreover, Proposition \ref{le0503-1} implies the existence of codebooks $\cC_2(N)$, $N\ge1$, with $\lim_{N\to \infty} \frac1N |\cC_2(N)|=0$ and
$$
\lim_{N\to\infty} \int \vphi(N^{1/d} d(x,\cC_2(N)))\,d\tilde\mu_s=0.
$$
Finally, Lemma \ref{le0324-1} (applied to $\tilde
X=1_{B(0,r_J)^c}(X)\cdot X$) yields the existence of codebooks
$\cC_3(N)$, $N\ge1$, for which
$$
\limsup_{N\to\infty} \frac1N |\cC_3(N)| \le \sum_{n=J}^\infty \alpha_n < \eps
$$
and
$$
\limsup_{N\to\infty} \int \vphi(N^{1/d} d(x,\cC_3(N)))\,d\mu_u\le \IE[\Psi(\|X\|)]\sum_{n=J}^\infty \frac1{\Psi(r_n)} \vphi(C\,\alpha_n^{-1/d} r_{n+1}).
$$
Now consider the codebooks $\cC(N)=\cC_1(N)\cup\cC_2(N)\cup\cC_3(N)$.
Due to the above estimates and (\ref{eq0503-1}), one has
\begin{align*}
\limsup_{N\to\infty} \int &\vphi(N^{1/d} d(x,\cC(N)))\,d\mu \\
&\le \int_{\IR^d} g(\tilde\xi(x))\,d\mu_c(x)+ \IE[\Psi(\|X\|)]\sum_{n=J}^\infty \frac1{\Psi(r_n)} \vphi(C\,\alpha_n^{-1/d} r_{n+1})<1,
\end{align*}
so that for sufficiently large $N$ it is true that
$
\|d(X,\cC(N))\|_{\vphi}\le N^{-1/d}$.
On the other hand,
$$
\limsup_{N\to\infty} \frac1N |\cC(N)| <(1+\eps)I+\eps
$$
and,  for sufficiently large $N$, it holds that $|\cC(N)|\le (I+\eps
I+\eps) N$. Consequently, it follows that for large $N$
$$
\delta((I+\eps I+\eps) N|X,\vphi)\le N^{-1/d}.
$$
Switching from $N$ to $M=(I+\eps I+\eps) N$ one obtains
$$
\delta(M|X,\vphi)\le (I+\eps I+\eps)^{1/d} \, M^{-1/d},
$$
for $M$ large. Since $\eps>0$ was arbitrary, it follows that
$$
\limsup_{M\to\infty} M^{1/d}\,\delta(M|X,\vphi)\le I^{1/d}
$$
and we proved the upper inequality.

In order to prove the lower bound we fix codebooks $\cC(N)$, $N\iN$, with at most $N$ elements and
$$
\limsup_{N\to\infty} N^{1/d}\, \delta(N|X,\vphi)\le I^{1/d}.
$$
By Proposition~\ref{prop4.2}, each accumulation point of the associated empirical measures
$$
\nu^N=\frac1N \sum_{\hat x\in\cC(N)} \delta_{\hat x}, \qquad N\iN,
$$
lies in $\cM$. In particular, \begin{align}\label{eq_size}
\lim_{N\to\infty} \frac{|\cC(N)|}{N}=1.
\end{align} Therefore, for any $\eps\in(0,1)$,
$$
\liminf_{N\to\infty} N^{1/d} \delta((1-\eps)N|X,\vphi)>I^{1/d}.
$$
Otherwise one could construct a sequence of codebooks $\cC(N)$, $N\iN$, as above which does not fulfil~(\ref{eq_size}). Switching from $N$ to $M=(1-\eps)N$ and letting $\eps\dto 0$ we obtain the lower bound.

The remaining properties of the minimizer $I$ were proved in Theorem~\ref{theo0315-1}.
\end{proof}

\bibliographystyle{apsr}

\end{document}